\documentclass{article}
\usepackage[utf8]{inputenc}
\usepackage{amsmath, amsfonts, amssymb, amsthm}
\usepackage{fullpage, mathabx, color, bbm}
\RequirePackage[colorlinks,citecolor=blue,urlcolor=blue,linkcolor=blue,linktocpage=true]{hyperref}
\usepackage{mathtools}
\mathtoolsset{showonlyrefs}
\usepackage{natbib}
\newtheorem{lem}{Lemma}
\newtheorem{thm}{Theorem}

\title{High-dimensional CLT for Sums of Non-degenerate Random Vectors: $n^{-1/2}$-rate}
\author{Arun Kumar Kuchibhotla and Alessandro Rinaldo}
\date{Carnegie Mellon University\\\texttt{\{arunku, arinaldo\}@cmu.edu}}

\begin{document}

\maketitle
\begin{abstract}
    In this note, we provide a Berry--Esseen bounds for rectangles in high-dimensions when the random vectors have non-singular covariance matrices. Under this assumption of non-singularity, we prove an $n^{-1/2}$ scaling for the Berry--Esseen bound for sums of mean independent random vectors with a finite third moment. The proof is essentially the method of compositions proof of multivariate Berry--Esseen bound from~\cite{senatov2011normal}. Similar to other existing works~\citep{kuchibhotla2018high,fang2020high}, this note considers the applicability and effectiveness of classical CLT proof techniques for the high-dimensional case. 
\end{abstract}
\section{Introduction}\label{sec:introduction}


The accuracy of the central limit theorem in growing and even in infinite dimensions is a classic topic in probability theory that has been extensively studied since at least the 1950's; see~\cite{paulauskas2012approximation,MR643968} for a historical account of central limit theorems.

In recent years, there has been a renewed impetus to derive sharp Berry-Esseeen bounds for Gaussian and bootstrap approximations in Euclidean spaces with explicit dependence on the dimension as tools to demonstrate the validity and effectiveness of statistical inference for high-dimensional models. In a seminal contribution, \cite{bentkus03} proved a Berry-Esseen error rate of order $p^{7/4}n^{-1/2}$ over the class of convex sets for isotropic distributions with bounded third moments \citep[see][for improvements]{raic2019,fang2020large}. While this result in principle allows for a central limit theorem in growing dimensions, its applicability to high-dimensional problems is relatively limited. 
\cite{CCK13} obtained a Berry-Esseeen bound of order $(\log(np))^{7/8} n^{-1/8}$ for the smaller class of $p$-dimensional hyper-rectangles under appropriate conditions; also see~\cite{MR1115160}. This important result implies a central limit guarantee that holds even when the dimension $p$ is much larger then $n$ and has led to significant and  important application in high-dimensional and non-parametric statistics: see, e.g., \cite{belloni2018high}. Subsequently, several authors generalized the settings of \cite{CCK13}  and improved on the original rate  both in terms of dependence on the dimension $p$ and of the exponent for the sample size: see, in particular, \cite{deng2017beyond,CCK17,kuchibhotla2018high,koike2019high,koike2019notes}. For some time, it was conjecture that the best dependence on the sample size, demonstrated by~\cite{CCK17}, should be of order $n^{-1/6}$, matching  the optimal sample size dependence for general Banach spaces found  by \cite{bentkus1987}.

Recently, \cite{chernozhukov2019improved} managed to achieve a better dependence on the sample size of order $n^{-1/4}$ for sub-Gaussian vectors while only requiring $\log(ed) = o(n^{1/5})$. Next, \cite{fang2020high},~\cite{das2020central}, and~\cite{lopes2020central} succeeded in going beyond a $n^{-1/4}$ dependence.  The most noticeable difference of these latest contributions compared to the earlier papers is that the random vectors are assumed to be non-degenerate, i.e., the minimum eigenvalue of the covariance matrix is bounded away from zero.  In detail, \citet[Corollary 1.3]{fang2020high} using the Stein's method improved the dependence on the sample size to $n^{-1/3}$ and the requirement on the dimension to $\log(ed) = o(n^{1/4})$. \citet[Corollary 1.1]{fang2020high} proved an $n^{-1/2}$ dependence on the sample size and the dimension requirement of $\log(ed) = o(n^{1/3})$, when the random vectors are log-concave. 
\cite{das2020central} established a $n^{-1/2}$ dependence when the random vectors have independent and sub-Gaussian coordinates (among other assumptions); they also study the optimal dependence on the dimension. However, the assumptions of log-concavity and independence are fairly restrictive. To date, the best result is contained in the breakthrough contribution of \cite{lopes2020central}, who exhibited for this problem a nearly  $n^{-1/2}$ rate of convergence only assuming non-degeneracy and independent and identically distributed sub-Gaussian vectors.

In this paper, we improve on the result of~\cite{lopes2020central} by deriving  an $n^{-1/2}$ rate under the weakest conditions possible and for independent but not necessarily identically distributed  sums, thus establishing the current state of the art. The proof is different from those exploited so far and relies on a non-trivial adaptation of the proof of~\cite{senatov1981uniform}, which itself is a method of compositions proof of the multivarite Berry--Esseen bound; see~\cite{bentkus2000accuracy}.
The main contributions of this paper is the derivation of Berry--Esseen bound for the Gaussian approximation of the distribution of sum of independent random vectors $X_1, \ldots, X_n$ over the class of $p$-dimensional hyper-rectangles
\begin{enumerate}
    \item assuming non-degeneracy of $X_1, \ldots, X_n$, i.e., $\mathbb{E}[X_iX_i^{\top}]$ is non-singular for all $i$;
    \item assuming finite third moment on the random vectors, i.e., $\max_i \mathbb{E}[\|X_i\|_{\infty}^3] < \infty$;
    \item allowing for non-identical distributions for $X_1, X_2, \ldots, X_n$.
\end{enumerate}
The last two aspects offer improvements over~\cite{lopes2020central}. Our motivation for considering random vectors with finite third moment stems from ``multiplier random vectors'' that are commonly encountered in high-dimensional regression problems. For instance, in the linear regression model $Y_i = X_i^{\top}\beta_0 + \xi_i$, estimation and inference for $\beta_0$ depends crucially on $\sum_{i=1}^n \xi_i X_i\in\mathbb{R}^p$. If the errors are heavy-tailed in that either $\mathbb{E}[|\xi_i|^{3+\delta}] < \infty$ or $\mathbb{E}[|\xi_i|^3|X_i] < \infty$ holds with sub-Gaussian or bounded $X_i$'s, then $\mathbb{E}[|\xi_i|^3\|X_i\|_{\infty}^{3}] < \infty$. In other words, $\xi_iX_i, 1\le i\le n$ are not sub-Gaussian.  See, in particular, \cite{clt.projection}.


The article is organized as follows. In Section~\ref{sec:main-result-CLT}, we state our main CLT result along with some discussion of its optimality. In Section~\ref{subsec:sketch}, we provide a sketch of the proof for first reading. In Section~\ref{sec:discussion}, we compare our result to that of~\cite{lopes2020central}. 
In Section~\ref{sec:conclusion}, we summarize the article and discuss some future directions. In Section~\ref{sec:notation-aux-lemmas}, we provide  auxiliary lemmas used in our proof and the complete proof of our main result is given in Section~\ref{subsec:proof-of-Theorem}.  

The proofs of all the results except the main result are provided in the appendices (Appendix~\ref{appsec:repeat-Thm-Senatov}--\ref{appsec:proof-of-optimality-third-moment}).
\paragraph{Notation.} We use the metric between the distribution of two random vectors $U$ and $V$ in $\mathbb{R}^p$ defined by
\begin{equation}\label{eq:metric-rectangles}
\mu(U, V) := \sup_{r\in\mathbb{R}^p}|\mathbb{P}(U \preceq r) - \mathbb{P}(V \preceq r)|,
\end{equation}
where for two $p$-dimensional vectors $x=(x(1),\ldots,x(p))$ and $y=(y(1),\ldots, y(p))$, $x \preceq y$ signifies that $x(j) \leq y(j)$ for all $j$. It is clear that $\mu(U, V)$ satisfies the triangle inequality and $\mu(t U, t V) = \mu(U, V)$ for all $t > 0$ (see Lemma~\ref{lem:regular-homogenity} below). We also use the ideal metric of order $3$ between probability distribution  \citep[see][Section 2.10]{senatov2011normal}, which is defined by
\begin{equation}\label{eq:ideal-metric-3}
\zeta_3(U, V) ~:=~\sup_{\substack{f:\mathbb{R}^p\to\mathbb{R},\|\nabla^3f(x)\|_1 \le 1\forall x}}|\mathbb{E}[f(U)] - \mathbb{E}[f(V)]|.
\end{equation}
Here $\|\nabla^3f(x)\|_1 = \sup_{\|h\|_{\infty} \le 1}\left|\partial^3 f(x + th)/\partial t^3\right|\big|_{t = 0}.$ It should be mentioned that the classical ideal metric in Euclidean spaces is defined with respect to the Euclidean norm instead of $\|\cdot\|_1-\|\cdot\|_{\infty}$ norms. For the high-dimensional case, our formulation is more suitable.

Our bounds will be presented in terms of pseudo-moments. For any two random vectors $U, V\in\mathbb{R}^p$ with probability measures $P_U, P_V$, respectively, define the pseudo-moments of order $1$ and $3$ as, respectively,
\begin{equation}\label{eq:psuedo-moments}
\nu_1(U, V) = \int \|x\|_{\infty}|P_U - P_V|(dx),\quad\mbox{and}\quad \nu_3(U, V) = \int \|x\|_{\infty}^3|P_U - P_V|(dx).
\end{equation}
For any vector $x\in\mathbb{R}^p$, we use $x(j)\in\mathbb{R}^p$ to denote the $j$-th coordinate of $x$ and $\|x\|_{\infty}$ to denote the $\ell_{\infty}$-norm: $\max_{1\le j\le p}|x(j)|$. The vector $\mathbf{1}\in\mathbb{R}^p$ represents the vector of all $1$'s. When taking the logarithm of the dimension or of the sample size, we  write $\log(ep)$ or $\log(en)$ so that the bound would be non-zero even if $p = 1$ or $n = 1$. For any two vectors $a, b\in\mathbb{R}^p$, we write $a \preceq b$ to denote $a(j) \le b(j), 1\le j\le p$, which is coordinate-wise inequality. For a positive definite matrix $A\in\mathbb{R}^{p\times p}$, $\lambda_{\min}(A)$ denotes the minimum eigenvalue of $A$. For any event $B$, we use $\mathbbm{1}\{B\}$ to denote the corresponding indicator function, which is one if $B$ holds and zero otherwise.


\section{Main Result}\label{sec:main-result-CLT}
We now present the main result of this paper: a Berry--Esseen bound for the normalized sums of independent random vectors over the class of $p$-dimensional  hyper-rectangles.
\begin{thm}\label{thm:main-theorem-CLT}
Let $X_1, X_2, \ldots, X_n$ be independent, centered random vectors in $\mathbb{R}^p$. Define
\[
\underline{\sigma}^2 := \min_{1\le i\le n}\lambda_{\min}(\mathbb{E}[X_i X_i^{\top}])
\quad\mbox{and}\quad 
\sigma_{\min}^2 := \min_{1\le i\le n}\min_{1\le j\le p}\mathbb{E}[X_i^2(j)],
\]
as the smallest minimum eigenvalue of the covariance matrices $\mathbb{E}[X_i X_i^{\top}]$'s and the smallest minimum variance of the coordinates of the $X_i$'s, respectively.
Let $Y_1, Y_2, \ldots, Y_n$ be independent centered Gaussian random vectors in $\mathbb{R}^p$
such that  $\mathbb{E}[Y_iY_i^{\top}] = \mathbb{E}[X_iX_i^{\top}]$ for all $i$. Define the maximal pseudo-moments of the two sequence of random vectors as
\begin{equation}\label{eq:maximal-pseudo-moments}
\nu_{j,n} ~:=~ \max_{1\le i\le n}\nu_j(X_i, Y_i),\quad\mbox{for}\quad j = 1, 3.
\end{equation}
Then, there exists a universal constant $\mathfrak{C} \ge 1$ such that 
\begin{align*}
\mu\left(\frac{1}{\sqrt{n}}\sum_{i=1}^n X_i,\,\frac{1}{\sqrt{n}}\sum_{i=1}^n Y_i\right) &\le \frac{3}{\sqrt{n}} + \mathfrak{C}\frac{\nu_{1,n}}{\sqrt{n}\sigma_{\min}}\log(ep)\sqrt{\log(pn)}\\ 
&+\mathfrak{C}\frac{\nu_{3,n}\log^2(ep)\sqrt{\log(pn)}}{\sqrt{n}\sigma_{\min}\underline{\sigma}^2}\left\{\frac{\sigma_{\min}/\underline{\sigma}}{\log(ep)} + \log\left(1 + \frac{\underline{\sigma}^3\sqrt{n/\log^3(ep)}}{2\mathfrak{C}\nu_{3,n}}\right)\right\}.
\end{align*}
\end{thm} 
The most important aspect of Theorem~\ref{thm:main-theorem-CLT} is that the dependence on the sample size in the bound is at most $\sqrt{\log(en)/n}$. The exact dependence on the dimension $p$ depends on the growth rate of the quantities $\nu_{1,n}$ and $\nu_{3,n}$ defined in~\eqref{eq:maximal-pseudo-moments}. When the random vectors are sub-Gaussian, then $\nu_{j,n} \lesssim \log(np)^{j/2}$, $j \in \{1,3\}$.

Note that the upper bound from Theorem~\ref{thm:main-theorem-CLT} is finite only when $\max_i \nu_3(X_i, Y_i) < \infty$ and this requires the random vectors $X_i$ to have a finite third moment. A dependence of $n^{-1/2}$ on the sample size cannot be obtained if the random vectors have less than three moments. Indeed, it is known that, even in the univariate case, with only $2 + \delta$ moments (for $0 < \delta \le 1$), the scaling in $n$ can at best be $n^{-\delta/2}$; see~\cite{katz1963note},~\cite{heyde1967influence}, and~\citet[Section 4.6]{senatov2011normal} for details. In this sense, the assumption of $\nu_3(X_i, Y_i) < \infty$ is the weakest possible to achieve a $n^{-1/2}$ rate. If we only assume $\mathbb{E}[\|X_i\|_{\infty}^{2 + \delta}] < \infty$ (for $\delta\in(0,1]$), then the calculations from~\cite{senatov1981uniform} can be used to prove an $n^{-\delta/2}$ rate in Theorem~\ref{thm:main-theorem-CLT}. 

Regarding the proof technique, we adapt the arguments used in the proof of Theorem 5.1.1 of~\cite{senatov2011normal}, which is a multivariate Berry-Esseen bound using the method of compositions~\citep{sweeting1977speeds,bergstrom1945central,senatov1981uniform}. 
Specifically, Theorem 5.1.1 of~\cite{senatov2011normal} consider general probability metrics $\mu(\cdot, \cdot)$ satisfying the four conditions stated there. Although we can verify all these assumptions for our metric~\eqref{eq:metric-rectangles}, the assumption (5.1.2) involves the total variation distance and cannot be controlled effectively in the high-dimensional case. For this reason, we provide an alternative to (5.1.2) of~\cite{senatov2011normal} as Lemma~\ref{lem:TV-alternative-lemma} in Section~\ref{sec:notation-aux-lemmas}. (For the convenience of the reader, we repeat the statement of Theorem 5.1.1 of~\cite{senatov2011normal} in Appendix~\ref{appsec:repeat-Thm-Senatov}.) 



\paragraph{Optimality in the Dependence of $\sigma_{\min}$ and $\underline{\sigma}$.} It is of interest to investigate if the dependence on $\sigma_{\min},\, \underline{\sigma}$ in the bound of Theorem~\ref{thm:main-theorem-CLT} can be replaced by some other characteristic of the covariance matrices. Below, we show  that in the worst case the dependence on the minimum eigenvalue or the minimum variance cannot be removed. This property has been discussed by several authors and in particular by~\cite{senatov1986four}. Our result below is also proved using the same construction. 
\begin{thm}\label{thm:optimal-dependence-sigma-min}
Suppose for any $p \ge 1$ and a sequence of independent and identically distributed random vectors $X_1, \ldots, X_n\in\mathbb{R}^p$, we have a bound of the form
\[
\mu\left(\frac{1}{\sqrt{n}}\sum_{i=1}^n X_i,\,\frac{1}{\sqrt{n}}\sum_{i=1}^n Y_i\right) ~\le~ \frac{\nu_1(X, Y)}{\sqrt{n}\sigma(\mathbb{E}[XX^{\top}])}\log^{\alpha_1}(ep) ~+~ \frac{\nu_3(X, Y)}{\sqrt{n}\sigma^3(\mathbb{E}[XX^{\top}])}\log^{\alpha_2}(ep),
\]
for some $\alpha_1, \alpha_2 \ge 0$ and functional $\sigma(\mathbb{E}[XX^{\top}])$ of the covariance matrix of $X$. 
Then there exists a distribution for $X$ such that $\sigma(\mathbb{E}[XX^{\top}])$ is of the same order as (a) $\sigma_{\min}^{1/3}\underline{\sigma}^{2/3}$, (b) $\sigma_{\min}$, and (c) $\underline{\sigma}$.
\end{thm}
\begin{proof}
See Appendix~\ref{appsec:proof-of-theorem-optimal-dependence-sigma-min} for a proof. The proof uses the probability distribution constructed in Example 1 of~\cite{senatov1986four}. 
\end{proof}

\paragraph{Optimality in the Dependence on $\mathbb{E}[\|X\|_{\infty}^3]$.} We now consider the dependence on the third moments of the random vectors. The ideas of~\cite{nagaev1976estimate} coupled with the example 1 of`\cite{senatov1986four} yields the following result.
\begin{thm}\label{thm:optimal-dependence-on-third-moment}
There exists a distribution $P_X$ on $\mathbb{R}^p$ such that the following holds true: if $X_1, X_2, \ldots, X_n\in\mathbb{R}^p$ are independent random vectors from $P_X$, then there exists a constant $C$ (independent of $p, n$) such that
\[
\mu\left(\frac{1}{\sqrt{n}}\sum_{i=1}^n X_i,\,\frac{1}{\sqrt{n}}\sum_{i=1}^nY_i\right) ~\ge~ C\frac{n^{-1}\sum_{i=1}^n \mathbb{E}[\|X_i\|_{\infty}^3]}{\sqrt{n}\sigma_{\min}\underline{\sigma}^2},
\]
whenever $\log^3(ep) \le n\sigma_{\min}^2\underline{\sigma}^4$.
\end{thm}
\begin{proof}
See Appendix~\ref{appsec:proof-of-optimality-third-moment} for a proof. This shows that our bound from Theorem~\ref{thm:main-theorem-CLT} cannot be improved except for poly-logarithmic factors. This result can also be derived from Proposition 1.1 of~\cite{fang2020high} by taking the coordinate random variables there to be sub-Gaussian.
\end{proof}

\subsection{Sketch of the Proof of Theorem~\ref{thm:main-theorem-CLT}}\label{subsec:sketch}
In this section, we provide an outline of the proof of Theorem~\ref{thm:main-theorem-CLT}. The basic structure is the same as Theorem 5.1.1 of~\cite{senatov2011normal}. In the sketch and the actual proof, we use the following notation: for any $1\le k\le k'\le n$,
\begin{equation}\label{eq:notation-sums}
S_{k:k'}(X) = \frac{1}{\sqrt{n}}\sum_{i=k}^{k'} X_i\quad\mbox{and}\quad S_{k:k'}(Y) = \frac{1}{\sqrt{n}}\sum_{i=k}^{k'} Y_i.
\end{equation} 
In this notation, Theorem~\ref{thm:main-theorem-CLT} requires bounding
\[
\sup_{r\in\mathbb{R}^p}\left|\mathbb{P}\left(S_{1:n}(X) \preceq r\right) - \mathbb{P}(S_{1:n}(Y) \preceq r)\right|.
\]
\paragraph{Smoothing and Lindeberg Swapping:} The quantity we want to control concerns expectations of indicator functions, which are not smooth. For this reason, most proofs of CLTs apply a smoothing to repalce indicator functions by smooth functions. A simple way to smooth a non-smooth function by Gaussian convolution: $\mathbb{E}[\mathbbm{1}\{x + \varepsilon Z \preceq r\}]$ is a smooth function of $x$ and approximates the indicator function $\mathbbm{1}\{x \preceq r\}$; here $Z$ is a standard Gaussian random vector in $\mathbb{R}^p$. This smoothing comes at the cost of an additional remainder term, which, according to Lemma~\ref{lem:smoothing-inequality} in Section~\ref{sec:notation-aux-lemmas} is of the form
\[
\mu(S_{1:n}(X), S_{1:n}(Y)) ~\le~ C\mu(S_{1:n}(X) + \varepsilon Z,\, S_{1:n}(Y) + \varepsilon Z) + \frac{C}{\sigma_{\min}}\varepsilon\log(ep).
\]
Setting $m = [n/2]$ and using the triangle inequality,
\begin{align*}
\mu(S_{1:n}(X) + \varepsilon Z, S_{1:n}(Y) + \varepsilon Z) &\le \sum_{j=0}^m \mu(S_{1:n-j-1}(X) + S_{n-j:n}(Y) + \varepsilon Z,\, S_{1:n-j}(X) + S_{n-j+1:n}(Y) + \varepsilon Z)\\ 
&\qquad+ \mu(S_{1:n-m-1}(X) + S_{n-m:n}(Y) + \varepsilon Z,\, S_{1:n}(Y) + \varepsilon Z).
\end{align*}
The classical Lindeberg swapping methods takes $m = n$,  leading to an $n^{-1/6}$ rate; see Theorem 2.1 of~\cite{bentkus2000accuracy}. \cite{lopes2020central} also relies on Lindeberg swapping but does not deploy the smoothing inequality of Lemma~\ref{lem:smoothing-inequality}. In addition, the $m$ is set differently.

Note that the decomposition performed above is not symmetric with respect to the indices of the random vectors, although the left hand side is invariant to permutation of the indices. One way to avoid this asymmetry is to apply Lindeberg swapping for ``interpolated'' random vectors as in~\cite{knowles2017anisotropic}; also see the discussion in pages 27-28 of~\cite{tao2017least}. Another way is to take the average of the right hand side bound over all the permutations of indices as done in~\cite{deng2017beyond,deng2020slightly}. We will not pursue this direction in this paper, for simplicity.
\paragraph{Application of Regularity and Smoothness:} To bound the metric between $S_{1:n-m-1}(X) + S_{n-m:n}(Y) + \varepsilon Z$ and $S_{1:n}(Y) + \varepsilon Z$, we use the fact that $S_{n-m:n}(Y) + \varepsilon Z$ can be decomposed in to sum of two independent Gaussians, one of which having a scaled identity as the covariance matrix. This is possible because
\[
N(0, \Xi) \overset{d}{=} N(0, \lambda_{\min}(\Xi)I_p) + N(0, \Xi - \lambda_{\min}(\Xi)I_p),
\]
where the two Gaussians on the right hand side are independent. With this decomposition, we can write 
\begin{align*}
&\mu(S_{1:n-m-1}(X) + S_{n-m:n}(Y) + \varepsilon Z,\, S_{1:n}(Y) + \varepsilon Z)\\
&\qquad= \mu(S_{1:n-m-1}(X) + \varepsilon_{m+1}Z + W,\, S_{1:n-m+1}(Y) + \varepsilon_{m+1} Z + W),
\end{align*}
where $\varepsilon_{m+1}^2 = \varepsilon^2 + (m+1)\underline{\sigma}^2/n$ is the minimum eigenvalue of $S_{n-m:n}(Y) + \varepsilon Z$ and $Z$ and $W$ are independent normal random variables. Now we use the intuition that adding noise to two random vectors makes it difficult to distinguish them. Thus, under a sufficiently regular metric (such as the one we consider), the distance between the distribution of $U + R$ and $V + R$ is less than between the distributions of $U$ and $V$; this is proved in Lemma~\ref{lem:regular-homogenity}. This allows us to ignore the random vector $W$. Furthermore, $\mu(U + \varepsilon Z, V + \varepsilon Z)$ involves expectations of smooth functions, thanks to the smoothness from the Gaussian smoothing. Thus, through a Taylor series expansion, we are able to relate $\mu(U + \varepsilon Z, V + \varepsilon Z)$ to an ideal metric of order $3$ \citep[see][Section 2.10]{senatov2011normal}; this is proved in Lemma~\ref{lem:mu-zeta-combination}. 
Hence Lemmas~\ref{lem:regular-homogenity} and~\ref{lem:mu-zeta-combination} (in Section~\ref{sec:notation-aux-lemmas}) yield
\[
\mu(S_{1:n-m-1}(X) + S_{n-m:n}(Y) + \varepsilon Z,\, S_{1:n}(Y) + \varepsilon Z) \le C(\log(ep))^{3/2}\frac{\zeta_3(S_{1:n-m+1}(X), S_{1:n-m+1}(Y))}{\varepsilon_{m+1}^3}.
\]
Because $m = [n/2]$, $\varepsilon_{m+1}^2 \ge \underline{\sigma}^2/2$. Hence,
\begin{align}
    \mu(S_{1:n}(X) + \varepsilon Z, S_{1:n}(Y) + \varepsilon Z) &\le \sum_{j=0}^m \mu(S_{1:n-j-1}(X) + S_{n-j:n}(Y) + \varepsilon Z,\, S_{1:n-j}(X) + S_{n-j+1:n}(Y) + \varepsilon Z)\nonumber\\
    &\quad+ C(\log(ep))^{3/2}\frac{\zeta_3(S_{1:n-m+1}(X), S_{1:n-m+1}(Y))}{\underline{\sigma}^3}\label{eq:second-term-done}
\end{align}
In~\cite{lopes2020central}, the author controls the first term using concentration inequality without using the Lindeberg swapping. He uses Lindeberg swapping for the second term above (instead of Lemma~\ref{lem:mu-zeta-combination}).
\paragraph{Application of the Recursion Lemma:} To bound the sum on the right hand side of~\eqref{eq:second-term-done}, we note that
\begin{equation}
\begin{split}
&\mu(S_{1:n-j-1}(X) + S_{n-j:n}(Y) + \varepsilon Z,\, S_{1:n-j}(X) + S_{n-j+1:n}(Y) + \varepsilon Z)\\ 
&\qquad= \mu\left(S_{1:n-j-1}(X) + \frac{Y_{n-j}}{\sqrt{n}} + S_{n-j+1:n}(Y) + \varepsilon Z,\,S_{1:n-j-1}(X) + \frac{X_{n-j}}{\sqrt{n}} + S_{n-j+1:n}(Y) + \varepsilon Z\right).
\end{split}
\end{equation}
Because of $\varepsilon Z$ in both sums, this is the difference between the expectations of smooth functions. However, if we just use a Taylor series expansion, then the resulting bound does not yield an $n^{-1/2}$ rate. For this, we also use the fact that both sums have $S_{1:n-j+1}(X)$ in common and that $S_{1:n-j+1}(X)$ is already close to a Gaussian $S_{1:n-j+1}(Y)$. Importantly, this results in a recursion, as formalized in Lemma~\ref{lem:TV-alternative-lemma} of Section~\ref{sec:notation-aux-lemmas}. 

For $1\le j\le m$, we use inequality~\eqref{eq:zeta-3-bound} of Lemma~\ref{lem:TV-alternative-lemma}. For $j = 0$, we use inequality~\eqref{eq:zeta-1-bound}. The choice of different inequalities from Lemma~\ref{lem:TV-alternative-lemma} is done to keep the dependence of $\varepsilon$ in the upper bound of $\sum_{j=0}^m\cdots$ as $\varepsilon^{-1}$. This is one of the important reasons why this proof leads to an $n^{-1/2}$ rate. Substituting these inequalities, we obtain
\begin{align*}
    &\sum_{j=0}^m \mu(S_{1:n-j-1}(X) + S_{n-j:n}(Y) + \varepsilon Z,\, S_{1:n-j}(X) + S_{n-j+1:n}(Y) + \varepsilon Z)\\ 
    &\quad\le \frac{C\nu_{1,n}\sqrt{\log(ep)}}{\varepsilon\sqrt{n}}\mu(S_{1:n-1}(X), S_{1:n-1}(Y))\\
    &\qquad+ C(\log(ep))^{3/2}\sum_{j=1}^m \mu(S_{1:n-j+1}(X), S_{1:n-j+1}(Y))\frac{\nu_3(X_{n-j}, Y_{n-j})}{\varepsilon_j^3n^{3/2}} + F_n(\varepsilon),
\end{align*}
for a function $F_n(\varepsilon)$ that weakly depends on $\varepsilon$ such that $F_n(1/\sqrt{n})$ converges to zero at an $n^{-1/2}$ rate. 
\paragraph{Wrap-up and Induction:} Using this bound in~\eqref{eq:second-term-done} yields
\begin{align*}
    \mu(S_{1:n}(X), S_{1:n}(Y)) &\le \frac{C\nu_{1,n}\sqrt{\log(ep)}}{\varepsilon\sqrt{n}}\mu(S_{1:n-1}(X), S_{1:n-1}(Y))\\
    &\qquad+ C(\log(ep))^{3/2}\sum_{j=1}^m \mu(S_{1:n-j+1}(X), S_{1:n-j+1}(Y))\frac{\nu_3(X_{n-j}, Y_{n-j})}{\varepsilon_j^3n^{3/2}} + F_n(\varepsilon)\\ 
    &\qquad+ C(\log(ep))^{3/2}\frac{\zeta_3(S_{1:n-m+1}(X), S_{1:n-m+1}(Y))}{\underline{\sigma}^3} + \frac{C}{\sigma_{\min}}\varepsilon\log(ep).
\end{align*}
This is a recursive inequality for the sequence $\{\mu(S_{1:k}(X), S_{1:k}(Y)): 1\le k\le n\}$. This should be compared to inequality (5.1.8) of~\cite{senatov2011normal}. 
The rest of the proof follows the same steps as in~\cite{senatov2011normal}. The basic idea is to hypothesize a bound on $\mu(S_{1:k}(X), S_{1:k}(Y))$ of the form $\tau(\{X_i\}, \{Y_i\})/\sqrt{n}$, substitute this in the recursive inequality and then find the right choice of $\tau(\{X_i\}, \{Y_i\})$. We found the induction calculations to be much simpler in~\cite{senatov2011normal} compared to those in~\cite{lopes2020central}, although the basic strategy remains the same.
\section[Discussion of Miles' paper]{Comparison with the Result of~\cite{lopes2020central}}\label{sec:discussion}
In this section, we provide a comparison of our result with that of~\cite{lopes2020central}. 
Let us recall Theorem 2.1 of~\cite{lopes2020central}. Assuming independent and identically distributed random vectors $X_1, X_2, \ldots, X_n$, Theorem 2.1 of~\cite{lopes2020central} states that
\begin{equation}\label{eq:lopes-result}
\mu\left(\frac{1}{\sqrt{n}}\sum_{i=1}^n X_i,\, \frac{1}{\sqrt{n}}\sum_{i=1}^n Y_i\right) \le C\frac{\nu^{5/2}}{\rho^{3/2}}\frac{\log^4(pn)\log(en)}{\sqrt{n}},
\end{equation}
where $\nu = \max_{1\le j\le p}\|X_i(j)/\sqrt{\mbox{Var}(X_i(j))}\|_{\psi_2}$ is the sub-Gaussian norm of scaled coordinates and $\rho$ is the minimum eigenvalue of the correlation matrix of $X$.

The bound in Theorem~\ref{thm:main-theorem-CLT})  compared favourably to~\eqref{eq:lopes-result} (Theorem 2.1 of~\cite{lopes2020central}). As pointed out before, we do not require sub-Gaussianity of the random vectors $X_1, \ldots, X_n$ or even identical distributions for $X_1, \ldots, X_n$. Our result is written in terms of pseudo-moments~\eqref{eq:psuedo-moments} and, because $$\nu_1(X_i, Y_i) \le \mathbb{E}[\|X_i\|_{\infty} + \|Y_i\|_{\infty}]\quad\mbox{and}\quad \nu_3(X_i, Y_i) \le \mathbb{E}[\|X_i\|_{\infty}^3 + \|Y_i\|_{\infty}^3],$$
the result can be written in terms of the moments of $X_i$'s. Under sub-Gaussianity of $X_i(j)$'s, we get $\mathbb{E}[\|X_i\|_{\infty}^3] \lesssim (\log(ep))^{3/2}$ and $\mathbb{E}[\|X_i\|_{\infty}] \lesssim \sqrt{\log(ep)}$; the constants suppressed here depend on the sub-Gaussian norm of $X_i(j)$'s. Hence the upper bound from Theorem~\ref{thm:main-theorem-CLT} becomes of order $\log^4(epn)\log(n)/\sqrt{n}$, which matches the rate from Theorem 2.1 of~\cite{lopes2020central}. Examples can be constructed where the coordinates of the random vectors are sub-Gaussian but our bound is significantly better than that of~\cite{lopes2020central}. For instance, consider independent and identically distributed random vectors $X_1, \ldots, X_n\in\mathbb{R}^p$ with independent coordinates such that
\[
X_i(1), \ldots, X_i(p-1)\overset{iid}{\sim} N(0, 1)\;\,\mbox{and}\;\, \mathbb{P}(X_i(p) = 0) = 1 - \frac{1}{\gamma}, \mathbb{P}(X_i(p) = -\gamma^{1/3}) = \mathbb{P}(X_i(p) = \gamma^{1/3}) = \frac{1}{2\gamma}.
\]
In this case $\nu_3(X_i, Y_i)/(\sigma_{\min}\underline{\sigma}^2) \le \gamma^{1/2}(\log(ep))^{3/2}$ and hence our bound is $\gamma^{1/2}(\log(epn))^{4}/\sqrt{n}$. However, the sub-Gaussian norm of $X_i(p)/\sqrt{\mbox{Var}(X_i(p))}$ is of order $\gamma^{1/2}$. Hence, the bound from~\cite{lopes2020central} becomes $\gamma^{5/4}\log^4(pn)\log(n)/\sqrt{n}$ which is sub-optimal for $\gamma$ diverging with $n$. Interestingly, for this example, Corollary 1.3 of~\cite{fang2020high} is still optimal although it uses the sub-Gaussian norm. 

It should be mentioned that the result of~\cite{lopes2020central} is given in terms of the minimum eigenvalue of the correlation matrix but the sub-Gaussian norm of the coordinates of $X_i$'s is given with coordinates normalized by their standard deviation. In comparison, our bound is written in terms of the minimum eigenvalue of the covariance matrix and if the minimum and maximum variances of the coordinates of $X_i$'s are of the same order, then our bound can be converted to only use the minimum eigenvalue of the correlation matrix. Finally, our bound distinguishes the dependence on minimum eigenvalue and minimum variance. In particular, if $\underline{\sigma} = \sigma_{\min}/\log(ep)$, then our bound becomes $\nu_{3,n}\log^4(ep)\sqrt{\log(pn)}\log(en)/\sqrt{n}$ while Theorem 2.1 of~\cite{lopes2020central} becomes $\nu^{5/2}\log^7(pn)\log(en)/\sqrt{n}$. (In case $\mbox{Var}(X_i(j)) = \sigma_{\min}^2$ for all $1\le j\le p$, then $\rho=(\underline{\sigma}/\sigma_{\min})^2$.)
\paragraph{Remark.}
In comparing the proof techniques, we note that the use of sub-Gaussian concentration inequality in Lemma 3.1 of \cite{lopes2020central} makes it difficult to derive an $n^{-1/2}$ rate of convergence when the random vectors involved admit only three finite moments. For instance, if one replaces the sub-Gaussian concentration bound used in Lemma 3.1 of \cite{lopes2020central}  by a Fuk-Nagaev-type inequality~\citep[see, e.g., Theorem 3]{einmahl2008characterization}, then it can be shown that the analog of Lemma 3.1 of~\cite{lopes2020central} yields an $n^{-q/(2q+2)}$ dependence on the sample size when $\mathbb{E}[\|X_i\|_{\infty}^q] < \infty$ for some $q \ge 2$. As a result, a dependence of $n^{-1/2}$ on the sample size cannot hold  for any finite $q$ and will require $q = \infty$. (Fortunately, one can allow for arbitrary sub-Weibull random vectors and still get an $n^{-1/2}$ rate.) The main merit of our proof is the avoidance of bounds based on concentration inequalities, which demand for conditions stronger than necessary.

We reamrk that several other papers also rely on the sub-Gaussian assumption and the corresponding concentration inequalities, although not in the same context. See~\citet[Lemma 6.2]{chernozhukov2019improved} and~\citet[Eq. (2.15)]{fang2020high} for some examples.
\section{Conclusions and Future Directions}\label{sec:conclusion}
In this note, we established  a high-dimensional Berry--Esseen bound for the sum of independent (but possibly non-identically distributed) random vectors over the class of hyper-rectangles. Our bound only depends on pseudo-moments of orders 1, 3, and hence it is finite whenever the random vectors have finite third moments. This allows for a significant sharpening of the recent result of~\cite{lopes2020central} and is essentially not improvable.

There are, however, several aspects of our analysis that warrants further investigation.
\begin{enumerate}
    \item Our result allows for non-identically distributed random vectors, but the dependence on the pseudo-moments $\nu_1(X_i, Y_i), \nu_3(X_i, Y_i)$ is not ``right.'' Berry--Esseen bounds for non-identically distributed random vectors often only involve $n^{-1}\sum_{i=1}^n \nu_1(X_i, Y_i)$ and $n^{-1}\sum_{i=1}^n \nu_3(X_i, Y_i)$. Further, the dependence on the minimum eigenvalue and minimum variance are also supposed to be with respect to the matrix $n^{-1}\sum_{i=1}^n \mathbb{E}[X_iX_i^{\top}]$. The ideas of~\cite{tao2017least},~\cite{deng2017beyond} and~\cite{deng2020slightly} are of interest in this context.
    \item Our result and that of~\cite{lopes2020central} require $\log(epn) = o(n^{1/8})$ for the bound to converge to zero, even in the most ``favourable'' case of sub-Gaussian random vectors. In this sense, the results are  sub-optimal in terms of the dependence on the dimension. It is now known from~\cite{fang2020high} and~\cite{das2020central} that one needs $\log(ep) = o(n^{1/3})$ for log-concave random vectors and $\log(ep) = o(n^{1/2})$ for vectors with independent (and symmetric, sub-Gaussian) coordinates. Without any log-concave/independence restrictions,~\citet[Corollary 1.3]{fang2020high} only requires $\log(ep) = o(n^{1/4})$, although yielding an $n^{-1/3}$ dependence on $n$. The optimality of the exponent of $\log(ep)$ in Theorem~\ref{thm:main-theorem-CLT} is an interesting future direction to consider.
    \item Our result proves an $n^{-1/2}$ scaling, assuming finite three moments. For an $n^{-1/2}$ scaling, even in the univariate case, the condition of three finite moments for the random variables is necessary. With only $2 + \delta$ moments (for $0 < \delta \le 1$), the scaling in $n$ can at best be $n^{-\delta/2}$; see~\cite{katz1963note},~\cite{heyde1967influence}, and~\citet[Section 4.6]{senatov2011normal} for details. High-dimensional CLT under $(2 + \delta)$-moments was only considered in~\cite{kuchibhotla2018high} albeit with a sub-optimal scaling in the sample size. It would be of interest to study the bounds in our setting when only $\mathbb{E}[\|X_i\|_{\infty}^{2 + \delta}] < \infty$. This again would be useful to understand the CLT for multiplier random vectors discussed after the contributions in Section~\ref{sec:introduction}.  
\end{enumerate}

\section{Proof of Theorem~\ref{thm:main-theorem-CLT}}\label{sec:proof-of-the-main-theorem}
In this section, we provide a detailed proof of Theorem~\ref{thm:main-theorem-CLT}. We will first provide the auxiliary lemmas used in the proof. The sketch of the proof is given in Section~\ref{subsec:sketch}.
\subsection{Auxiliary Lemmas}\label{sec:notation-aux-lemmas}
The following lemmas provide high-dimensional analogs of the conditions of Theorem 5.1.1 of~\cite{senatov2011normal}. See Appendix~\ref{appsec:repeat-Thm-Senatov} for the conditions of Theorem 5.1.1 of~\cite{senatov2011normal}. The relevance of these lemmas is discussed in the sketch presented in Section~\ref{subsec:sketch} and we advise the reader to look back at the sketch when reading the lemmas below. 

In some of these results, we use the high-dimensional anti-concentration inequality from~\cite{Naz03} and~\cite{chernozhukov2017detailed}, which implies that
\[
\mathbb{P}(r \preceq N(0, \Sigma) \preceq r + \delta\mathbf{1}) ~\le~ \frac{C\sqrt{\log(ep)}}{\min_{1\le j\le p}\Sigma_{jj}^{1/2}}\times\delta,
\]
for an absolute constant $C > 0$. This is a worst-case anti-concentration inequality and can be improved either using Theorem 10 of~\cite{deng2017beyond} or using Lemma 2.2 of~\cite{koike2019notes}. The latter uses special structure on the covariance matrix $\Sigma$. The exact dependence on the anti-concentration constant (the coefficient of $\delta$) in Lemmas~\ref{lem:smoothing-inequality}--~\ref{lem:mu-zeta-combination} can be obtained from the proofs. This can reduce the exponent of $\log(ep)$ in Theorem~\ref{thm:main-theorem-CLT}.
\begin{lem}[Smoothing Inequality]\label{lem:smoothing-inequality}
Suppose $V\sim N(0, \Sigma)$ and let $U$ be another random vector in $\mathbb{R}^p$. Then for any $\varepsilon > 0$ and a standard Gaussian random vector $Z\sim N(0, I_p)$, there exists a constant $C > 0$ such that
\[
\mu(U, V) ~\le~ \mu(U + \varepsilon Z, V + \varepsilon Z) + \frac{C}{\sigma_{\min}}\varepsilon\log(ep),
\]
where ${\sigma}^2_{\min} := \min_{1\le j\le p}\mathrm{Var}(V(j))$. 
\end{lem}
\begin{proof}
See Appendix~\ref{appsec:proof-of-lemma-smoothing-inequality} for a proof. This is proved based on Lemma 2.4 of~\cite{fang2020high}. This is a smoothing inequality that allows us to replace $\mathbb{P}(U \preceq r) = \mathbb{E}[\mathbbm{1}\{U \preceq r\}]$ by the expectation of a smooth function: $\mathbb{E}[\mathbbm{1}\{U + \varepsilon Z \preceq r\}] = \mathbb{E}[\varphi_{\varepsilon}(U)]$. Here $\varphi_{\varepsilon}(x) = \mathbb{P}(x + \varepsilon Z \preceq r)$ which is an infinitely differentiable function of $x$; see Lemma 2.3 of~\cite{fang2020high}.
\end{proof}
\begin{lem}[Recursion Lemma]\label{lem:TV-alternative-lemma}
Suppose $U, V, W$ are independent random vectors in $\mathbb{R}^p$ and $Z$ is a standard Gaussian random vector. Set
\[
\sigma_{\min}^2 := \min_{1\le j\le p}\mathrm{Var}(W(j)). 
\]
Then for a Gaussian random vector $W'$ with same mean and covariance as $W$, we have
\begin{equation}\label{eq:zeta-1-bound}
\begin{split}
\mu(W + U + \varepsilon Z,\, W + V + \varepsilon Z) &\le \frac{C\nu_1(U, V)\sqrt{\log(ep)}}{\varepsilon}\mu(W, W')\\ &\quad+ \nu_1(U, V)\left[\frac{C\log(ep)\sqrt{\log(pn)}}{\sigma_{\min}} + \frac{C}{\varepsilon pn}\right].
\end{split}
\end{equation}
Moreover, if $U, V$ match the first two moments, i.e., $\mathbb{E}[U - V] = 0$ and $\mathbb{E}[UU^{\top} - VV^{\top}] = 0$, then
\begin{equation}\label{eq:zeta-3-bound}
\begin{split}
\mu(W + U + \varepsilon Z,\, W + V + \varepsilon Z) &\le
\frac{C\nu_3(U, V)(\log(ep))^{3/2}}{\varepsilon^3}\mu(W, W')\\ &\quad+ \nu_3(U, V)\left[\frac{C\log^2(ep)\sqrt{\log(pn)}}{\varepsilon^2\sigma_{\min}} + \frac{C}{\varepsilon^3pn}\right].
\end{split}
\end{equation}
\end{lem}
\begin{proof}
See Appendix~\ref{appsec:proof-of-lemma-TV-alternative-lemma} for a proof. The proof of this is based on the calculations in Lemma 5.1 of~\cite{lopes2020central}. A similar result can be extracted from the proof of Theorem 2.1 of~\cite{bentkus2000accuracy}.
\end{proof}

\begin{lem}[Regularity and Homogeneity]\label{lem:regular-homogenity}
For independent random vectors $U, V, W$,
$\mu(U + W, V + W) \le \mu(U, V).$
Moreover, for any $t > 0$, $\mu(t U, t V) = \mu(U, V).$
\end{lem}
\begin{proof}
See Appendix~\ref{appsec:proof-of-lemma-regular-homogenity} for a proof. This proves that the metric $\mu$ is regular and homogeneous of order zero as described in~\citet[Section 2.3]{senatov2011normal}. Regularity (the first result here) implies that two distributions are hard to distinguish when convolved with a common distribution. This is an intuitive idea because adding noise to two samples makes it hard to distinguish the two samples. 
\end{proof}
\begin{lem}[Ideal Metric Lemma]\label{lem:mu-zeta-combination}
For independent random vectors $U, V$ and an independent standard Gaussian random vector $Z\sim N(0, I_p)$, 
\[
\mu(U + \varepsilon Z, V + \varepsilon Z) ~\le~ c(\log(ep))^{3/2}\frac{\zeta_3(U, V)}{\varepsilon^3},
\]
where $\zeta_3(\cdot, \cdot)$ is defined in~\eqref{eq:ideal-metric-3}.
\end{lem}
\begin{proof}
See Appendix~\ref{appsec:proof-of-lemma-mu-zeta-combination} for a proof. This basically follows from the fact that $\mu(U + \varepsilon Z, V + \varepsilon Z)$ concerns the difference between expectations of smooth functions that are infinitely differentiable. Using the bounds on the derivatives from~\citet[Lemma 2.3]{fang2020high} yields the result.
\end{proof}
\begin{lem}[Pseudo-moment Lemma]\label{lem:ideal-to-pseudo}
For independent random vectors $U, V$, and $c \in \mathbb{R}$, we have
\[
\zeta_3(c U, c V) ~=~ |c|^3\zeta_3(U, V) ~\le~ \frac{|c|^3}{6}\nu_3(U, V),
\]
where the second inequality holds only when $U, V$ have the same first two moments, i.e., $\mathbb{E}[U - V] = 0$ and $\mathbb{E}[UU^{\top} - VV^{\top}] = 0$.
\end{lem}
\begin{proof}
See Appendix~\ref{appsec:proof-of-lemma-ideal-to-pseudo} for a proof. The first equality is proving that $\zeta_3(\cdot, \cdot)$ is homogeneous of order $3$. 
\end{proof}

\subsection{Proof of Theorem~\ref{thm:main-theorem-CLT}}\label{subsec:proof-of-Theorem}
Throughout the proof, $C>0$ will denote a quantity independent of $p$ and of the distributions of $(X_1,\ldots,X_n)$ and of $(Y_1,\ldots,Y_n)$, whose value may change at each occurrence with the exception of the last inductive step of proof, where $C$ is fixed once and for all throughout the induction. Note that the result is clearly true for $n = 1, 2$ (because $\mu(\cdot, \cdot) \le 1$) and hence, we assume $n\ge3$.\\ 

\noindent{\bf Smoothing and Lindeberg Swapping}:
By Lemma~\ref{lem:smoothing-inequality}, we have that 
\begin{equation}\label{eq:eq-5.1.5}
\mu(S_{1:n}(X), S_{1:n}(Y)) ~\le~ C\mu(S_{1:n}(X) + \varepsilon Z,\, S_{1:n}(Y) + \varepsilon Z) + \frac{C}{\sigma_{\min}}\varepsilon\log(ep).
\end{equation}
Setting $m = [n/2]$ and the using triangle inequality, yields the bound 
\begin{equation}\label{eq:eq-5.1.6}
\begin{split}
\mu(S_{1:n}(X) + \varepsilon Z, S_{1:n}(Y) + \varepsilon Z) &\le \sum_{j=0}^m \mu(S_{1:n-j-1}(X) + S_{n-j:n}(Y) + \varepsilon Z,\, S_{1:n-j}(X) + S_{n-j+1:n}(Y) + \varepsilon Z)\\ 
&\qquad+ \mu(S_{1:n-m-1}(X) + S_{n-m:n}(Y) + \varepsilon Z,\, S_{1:n}(Y) + \varepsilon Z).
\end{split}
\end{equation}
We will bound the two terms in the last expression separately.\\ 

\noindent{\bf Control using the properties of the metric $\mu$}:
We  control the last term in~\eqref{eq:eq-5.1.6} using the regularity and homogeneity properties of the metric $\mu(\cdot,\cdot)$ as indicated in Lemma~\ref{lem:regular-homogenity}, and its relations with the ideal metric of order $3$ and the third pseudo-moment, elucidated in Lemma~\ref{lem:mu-zeta-combination}. We refer the reader to \citet[Chapter 2]{senatov2011normal} for details.
Firstly, note that
\begin{align*}
S_{n-m:n}(Y) + \varepsilon Z ~&\overset{d}{=}~ N\left(0,\,\frac{1}{n}\sum_{i=n-m}^n \mathbb{E}[Y_iY_i^{\top}] + \varepsilon^2I_p\right)\\ 
~&\overset{d}{=}~ N\left(0, \left\{\frac{(m+1)}{n}\underline{\sigma}^2 + \varepsilon^2\right\}I_p\right) + N\left(0, \frac{1}{n}\sum_{i=n-m}^n \mathbb{E}[Y_iY_i^{\top}] - \frac{(m+1)}{n}\underline{\sigma}^2I_p\right).
\end{align*}
The normal random vectors in the second equality are independent. Therefore, Lemma~\ref{lem:regular-homogenity} allows us to conclude
\[
\mu(S_{1:n-m-1}(X) + S_{n-m:n}(Y) + \varepsilon Z,\, S_{1:n}(Y) + \varepsilon Z) ~\le~ \mu(S_{1:n-m+1}(X) + \varepsilon_{m+1}Z,\, S_{1:n-m+1}(Y) + \varepsilon_{m+1}Z),
\]
where $\varepsilon_{m+1}^2 := \varepsilon^2 + (m+1)\underline{\sigma}^2/n$. Next, Lemma~\ref{lem:mu-zeta-combination} implies that
\begin{equation}\label{eq:second-term-bound-eq:5.1.5}
\begin{split}
\mu(S_{1:n-m+1}(X) + \varepsilon_{m+1}Z,\, S_{1:n-m+1}(Y) + \varepsilon_{m+1}Z) &\le C(\log(ep))^{3/2}\frac{\zeta_3(S_{1:n-m+1}(X), S_{1:n-m+1}(Y))}{\varepsilon^3_{m+1}}\\ 
&\le \frac{C(\log(ep))^{3/2}}{n^{3/2}\underline{\sigma}^3}\sum_{i=1}^{n-m+1} \nu_3(X_i, Y_i).
\end{split}
\end{equation}
The second inequality stems from the triangle inequality for $\zeta_3(\cdot, \cdot)$, the fact that our choice of $m$ implies that both $m$ and $n-m$ are equal, up to multiplicative constants, to $n$ and the bound 
\[
\zeta_3( n^{-1/2} X_i, n^{-1/2}Y_i) \leq C  \nu_3( n^{-1/2} X_i, n^{-1/2}Y_i) = C n^{-3/2} \nu_3(  X_i, Y_i),
\]
valid for all $i$ as shown in Lemma~\ref{lem:ideal-to-pseudo}; also, see Eq. (2.10.3) of~\cite{senatov2011normal}.\\ 

\noindent{\bf Control using a Gaussian bound for the metric $\mu$}:
We now turn to the first term in~\eqref{eq:eq-5.1.6}, which we control using Lemma~\ref{lem:TV-alternative-lemma}.
For each $j=1,\ldots,m$ we set $\varepsilon_j^2 = \varepsilon^2 + \underline{\sigma}^2j/n$ and bound the corresponding summand as follows:
\begin{equation}
\begin{split}
&\mu(S_{1:n-j-1}(X) + S_{n-j:n}(Y) + \varepsilon Z,\, S_{1:n-j}(X) + S_{n-j+1:n}(Y) + \varepsilon Z) \nonumber\\
&\qquad= \mu\left(S_{1:n-j-1}(X) + \frac{Y_{n-j}}{\sqrt{n}} + S_{n-j+1:n}(Y) + \varepsilon Z,\, S_{1:n-j-1}(X) + \frac{X_{n-j+1}}{\sqrt{n}} + S_{n-j+1:n}(Y) + \varepsilon Z\right) \nonumber\\
&\qquad \le \mu\left(S_{1:n-j+1}(X) + \frac{Y_{n-j}}{\sqrt{n}} + \varepsilon_{j}Z,\, S_{1:n-j+1}(X) + \frac{X_{n-j}}{\sqrt{n}} + \varepsilon_j Z\right), 
\end{split}
\end{equation}
where the last inequality follows from Lemma~\ref{lem:regular-homogenity} and the fact that for 
\[
S_{n-j+1:n}(Y) + \varepsilon Z ~\overset{d}{=}~ N\left(0, \varepsilon_j^2I_p\right) + N\left(0, \frac{1}{n}\sum_{i=n-j+1}^n \mathbb{E}[Y_iY_i^{\top}] - \frac{j}{n}\underline{\sigma}^2I_p\right).
\]
Inequality~\eqref{eq:zeta-3-bound} in Lemma~\ref{lem:TV-alternative-lemma} implies that $\mu\left(S_{1:n-j+1}(X) + {Y_{n-j}}{n^{-1/2}} + \varepsilon_{j}Z,\, S_{1:n-j+1}(X) + {X_{n-j}}{n^{-1/2}} + \varepsilon_j Z\right)$ is upper bounded by 
\begin{equation}\label{eq:application-alternative-TV-lemma}
\begin{split}
& \frac{C\nu_3(X_{n-j}/\sqrt{n}, Y_{n-j}/\sqrt{n})(\log(ep))^{3/2}}{\varepsilon^3_j}\mu(S_{1:n-j+1}(X), S_{1:n-j+1}(Y))\\ &\qquad\quad+ \beta(X_{n-j}/\sqrt{n}, Y_{n-j}/\sqrt{n})\left[\frac{C\log^2(ep)\sqrt{\log(pn)}}{\varepsilon^2_j\sigma_{\min}} + \frac{C}{\varepsilon^3_jpn}\right]\\
&\le \frac{C\nu_3(X_{n-j},Y_{n-j})(\log(ep))^{3/2}}{n^{3/2}\varepsilon^3_j}\mu(S_{1:n-j+1}(X), S_{1:n-j+1}(Y))\\ &\qquad\quad+ \frac{C\nu_3(X_{n-j}, Y_{n-j})}{n^{3/2}}\left[\frac{\log^2(ep)\sqrt{\log(pn)}}{\varepsilon_j^2\sigma_{\min}} + \frac{1}{\varepsilon_j^3pn}\right],
\end{split}
\end{equation}
where in the last step we have used again the fact that $\nu_3(cX, cY) = |c|^3\nu_3(X, Y)$, for all $c \in \mathbb{R}$.
For the summand indexed by $j = 0$, we use instead inequality~\eqref{eq:zeta-1-bound} in Lemma~\ref{lem:TV-alternative-lemma}  to obtain that
\begin{equation}\label{eq:application-2-alternative-TV-lemma}
\begin{split}
\mu(S_{1:n-1}(X) + S_{n:n}(Y) + \varepsilon Z,\, S_{1:n}(X) + \varepsilon Z) &\le \frac{C\nu_1(X_n, Y_n)\sqrt{\log(ep)}}{\varepsilon\sqrt{n}}\mu(S_{1:n-1}(X), S_{1:n-1}(Y))\\ &\qquad+ \frac{C\nu_1(X_n, Y_n)}{\sqrt{n}}\left[\frac{\log(ep)\sqrt{\log(pn)}}{\sigma_{\min}} + \frac{1}{\varepsilon pn}\right].
\end{split}
\end{equation}
Summing the inequalities~\eqref{eq:application-alternative-TV-lemma} and~\eqref{eq:application-2-alternative-TV-lemma} we obtain that
\begin{equation}\label{eq:combined-alternative-TV-lemma-implication}
\begin{split}
&\sum_{j=0}^m \mu(S_{1:n-j-1}(X) + S_{n-j:n}(Y) + \varepsilon Z,\, S_{1:n-j}(X) + S_{n-j+1:n}(Y) + \varepsilon Z)\\ &\quad\le \frac{C\nu_1(X_n, Y_n)\sqrt{\log(ep)}}{\varepsilon\sqrt{n}}\mu(S_{1:n-1}(X), S_{1:n-1}(Y))\\
&\quad\qquad+ \sum_{j=1}^m \frac{C\nu_3(X_{n-j},Y_{n-j})(\log(ep))^{3/2}}{n^{3/2}\varepsilon^3_j}\mu(S_{1:n-j+1}(X), S_{1:n-j+1}(Y))\\
&\quad\qquad+ \frac{C\nu_1(X_n, Y_n)}{\sqrt{n}}\left[\frac{\log(ep)\sqrt{\log(pn)}}{\sigma_{\min}} + \frac{1}{\varepsilon pn}\right]\\
&\quad\qquad+ \sum_{j=1}^m \frac{C\nu_3(X_{n-j}, Y_{n-j})}{n^{3/2}}\left[\frac{\log^2(ep)\sqrt{\log(pn)}}{\varepsilon_j^2\sigma_{\min}} + \frac{1}{\varepsilon_j^3pn}\right].
\end{split}
\end{equation}
Substituting~\eqref{eq:combined-alternative-TV-lemma-implication} and~\eqref{eq:second-term-bound-eq:5.1.5} in~\eqref{eq:eq-5.1.6} and then in~\eqref{eq:eq-5.1.5}, we arrive at the bound
\begin{equation}\label{eq:almost-final-recursive-inequality}
    \begin{split}
        \mu(S_{1:n}(X), S_{1:n}(Y)) ~&\le~
        \frac{C\nu_1(X_n, Y_n)\sqrt{\log(ep)}}{\varepsilon\sqrt{n}}\mu(S_{1:n-1}(X), S_{1:n-1}(Y))\\
        &\quad+ \sum_{j=1}^m \frac{C\nu_3(X_{n-j},Y_{n-j})(\log(ep))^{3/2}}{n^{3/2}\varepsilon^3_j}\mu(S_{1:n-j+1}(X), S_{1:n-j+1}(Y))\\
        &\quad+ \frac{C\nu_1(X_n, Y_n)}{\sqrt{n}}\left[\frac{\log(ep)\sqrt{\log(pn)}}{\sigma_{\min}} + \frac{1}{\varepsilon pn}\right]\\
        &\quad+ \sum_{j=1}^m \frac{C\nu_3(X_{n-j}, Y_{n-j})}{n^{3/2}}\left[\frac{\log^2(ep)\sqrt{\log(pn)}}{\varepsilon_j^2\sigma_{\min}} + \frac{1}{\varepsilon_j^3pn}\right]\\
        &\quad+ \frac{C(\log(ep))^{3/2}}{n^{3/2}\underline{\sigma}^3}\sum_{i=1}^{n-m+1} \nu_3(X_i, Y_i) + \frac{C}{\sigma_{\min}}\varepsilon\log(ep).
    \end{split}
\end{equation}
Bounding $\nu_3(X_i, Y_i)$ by $\nu_{3,n}$ and using the inequalities
\begin{equation}\label{eq:sum-of-epsilon_j's}
\begin{split}
\sum_{j=1}^m \frac{1}{\varepsilon_j^2} &\le n\int_0^{1/2} \frac{1}{(\varepsilon^2 + \underline{\sigma}^2t)}dt \le \frac{2n}{\underline{\sigma}^2}\log\left(1 + \frac{\underline{\sigma}}{\varepsilon}\right)\\
\sum_{j=1}^m \frac{1}{\varepsilon_j^3} &\le n\int_0^{1/2} \frac{1}{(\varepsilon^2 + \underline{\sigma}^2t)^{3/2}}dt \le \frac{2n}{\underline{\sigma}^{2}}\left[\frac{1}{\sqrt{\varepsilon^2}} - \frac{1}{\sqrt{\varepsilon^2 + \underline{\sigma}^2/2}}\right] \le \frac{2n}{\varepsilon\underline{\sigma}^2},
\end{split}
\end{equation}
we conclude that
\begin{equation}\label{eq:analog-eq:5.1.8}
    \begin{split}
        \mu(S_{1:n}(X), S_{1:n}(Y)) &\le \frac{C\nu_{1,n}\sqrt{\log(ep)}}{\varepsilon\sqrt{n}}\mu(S_{1:n-1}(X), S_{1:n-1}(Y))\\
        &\quad+ C(\log(ep))^{3/2}\sum_{j=1}^m \mu(S_{1:n-j+1}(X), S_{1:n-j+1}(Y))\frac{\nu_3(X_{n-j}, Y_{n-j})}{\varepsilon_j^3n^{3/2}}\\
        &\quad+ \frac{C\nu_{1,n}}{\sqrt{n}}\left[\frac{\log(ep)\sqrt{\log(pn)}}{\sigma_{\min}} + \frac{1}{\varepsilon pn}\right]\\
        &\quad+ \frac{C\nu_{3,n}\log^2(ep)\sqrt{\log(pn)}}{n^{1/2}\sigma_{\min}\underline{\sigma}^2}\log\left(1 + \frac{\underline{\sigma}}{\varepsilon}\right) + \frac{C\nu_{3,n}}{\varepsilon \underline{\sigma}^2pn^{3/2}}\\
        &\quad+\frac{C(\log(ep))^{3/2}}{n^{1/2}\underline{\sigma}^3}\nu_{3,n} + \frac{C}{\sigma_{\min}}\varepsilon\log(ep). 
    \end{split}    
\end{equation}

\noindent{\bf Induction Step:}
The relation in the last display provides a recursive inequality relating $\mu(S_{1:n}(X), S_{1:n}(Y))$ and $\mu(S_{1:n-1}(X), S_{1:n-1}(Y))$ that should  be compared with the analogous inequality  (5.1.8) of~\cite{senatov2011normal}. We now proceed to prove the result by induction. Below, the quantity $C>0$ is a  universal constant whose value maybe be chosen appropriately but in a manner that is independent of the level of the induction.
To set up the induction, we assume that
\begin{equation}\label{eq:induction-hypothesis}
    \mu(S_{1:k}(X), S_{1:k}(Y)) \le \frac{\tau(\{X_i\}, \{Y_i\})}{\sqrt{k}},\tag{$H_k$}
\end{equation}
for some $\tau(\cdot, \cdot)$ to be specified later. 
Because $\mu(\cdot, \cdot) \le 1$, this hypothesis is trivially verified for $k = 1$ if $\tau(\{X_i\}, \{Y_i\}) \ge 1$. Assume that the induction hypothesis~\eqref{eq:induction-hypothesis} holds true for $1 \le k \le n-1$. Then the right hand side of~\eqref{eq:analog-eq:5.1.8} can be upper bounded as
\begin{align*}
    \mu(S_{1:n}(X), S_{1:n}(Y)) &\le \frac{C\nu_{1,n}\sqrt{\log(ep)}}{\varepsilon\sqrt{n}}\left[\frac{\tau(\{X_i\}, \{Y_i\})}{\sqrt{n-1}}\right]\\
    &\quad+ C\nu_{3,n}(\log(ep))^{3/2}\sum_{j=1}^m \frac{1}{\varepsilon_j^3n^{3/2}}\left[\frac{\tau(\{X_i\}, \{Y_i\})}{\sqrt{n-j+1}}\right]\\
    &\quad+ \frac{C\nu_{1,n}}{\sqrt{n}}\left[\frac{\log(ep)\sqrt{\log(pn)}}{\sigma_{\min}} + \frac{1}{\varepsilon pn}\right]\\
    &\quad+ \frac{C\nu_{3,n}\log^2(ep)\sqrt{\log(pn)}}{n^{1/2}\sigma_{\min}\underline{\sigma}^2}\log\left(1 + \frac{\underline{\sigma}}{\varepsilon}\right) + \frac{C\nu_{3,n}}{\varepsilon \underline{\sigma}^2pn^{3/2}}\\
    &\quad+\frac{C(\log(ep))^{3/2}}{n^{1/2}\underline{\sigma}^3}\nu_{3,n} + \frac{C}{\sigma_{\min}}\varepsilon\log(ep).
\end{align*}
Because $n/3 \le n-j+1$ for all $1\le j\le m = [n/2]$, the second term above can be bounded using~\eqref{eq:sum-of-epsilon_j's} as
\begin{align*}
\sum_{j=1}^m \frac{1}{\varepsilon_j^3n^{3/2}}\left[\frac{\tau(\{X_i\}, \{Y_i\})}{\sqrt{n-j+1}}\right] \le C\left[\frac{\tau(\{X_i\}, \{Y_i\})}{\sqrt{n}}\right]\frac{1}{\varepsilon\underline{\sigma}^2\sqrt{n}}.
\end{align*}
Hence the induction hypothesis allows us to conclude for any $\varepsilon > 0$ and enlarging $C>0$ by a multiplicative factor  $\sqrt{n/(n-1)} \leq 2$,
\begin{equation}
    \begin{split}
        \mu(S_{1:n}(X), S_{1:n}(Y)) &\le C\left[\frac{\tau(\{X_i\}, \{Y_i\})}{\sqrt{n}}\right]\left\{\frac{\nu_{1,n}\sqrt{\log(ep)}}{\varepsilon\sqrt{n}} + \frac{\nu_{3,n}(\log(ep))^{3/2}}{\varepsilon\underline{\sigma}^2\sqrt{n}}\right\}\\
        &\quad+ \frac{C\nu_{1,n}}{\sqrt{n}}\left[\frac{\log(ep)\sqrt{\log(pn)}}{\sigma_{\min}} + \frac{1}{\varepsilon pn}\right]\\
        &\quad+ \frac{C\nu_{3,n}\log^2(ep)\sqrt{\log(pn)}}{n^{1/2}\sigma_{\min}\underline{\sigma}^2}\log\left(1 + \frac{\underline{\sigma}}{\varepsilon}\right) + \frac{C\nu_{3,n}}{\varepsilon \underline{\sigma}^2pn^{3/2}}\\
        &\quad+\frac{C(\log(ep))^{3/2}}{n^{1/2}\underline{\sigma}^3}\nu_{3,n} + \frac{C}{\sigma_{\min}}\varepsilon\log(ep).
    \end{split}
\end{equation}
Because this inequality holds for any $\varepsilon > 0$, we take (for some constant $\mathfrak{K}$ to be specified shortly)
\[
\varepsilon = \mathfrak{K}\left\{\frac{\nu_{1,n}\sqrt{\log(ep)}}{\sqrt{n}} + \frac{\nu_{3,n}(\log(ep))^{3/2}}{\underline{\sigma}^2\sqrt{n}}\right\}.
\]
(This choice comes from looking at the first term). From this choice of $\varepsilon$, we obtain
\begin{equation}\label{eq:some-intermediate-equation-induction}
    \begin{split}
        &\mu(S_{1:n}(X), S_{1:n}(Y))\\ &\quad\le \frac{C}{\mathfrak{K}}\left[\frac{\tau(\{X_i\}, \{Y_i\})}{\sqrt{n}}\right] + \frac{C\nu_{1,n}}{\sqrt{n}}\left[\frac{\log(ep)\sqrt{\log(pn)}}{\sigma_{\min}} + \frac{1}{\mathfrak{K}p\sqrt{n}\nu_{1,n}\sqrt{\log(ep)}}\right]\\
        &\quad\quad+ \frac{C\nu_{3,n}\log^2(ep)\sqrt{\log(pn)}}{n^{1/2}\sigma_{\min}\underline{\sigma}^2}\log\left(1 + \frac{\underline{\sigma}^3\sqrt{n}}{\mathfrak{K}\nu_{3,n}(\log(ep))^{3/2}}\right) + \frac{C}{\mathfrak{K}pn(\log(ep))^{3/2}}\\
        &\quad\quad+ \frac{C(\log(ep))^{3/2}}{n^{1/2}\underline{\sigma}^3}\nu_{3,n} + \frac{C\mathfrak{K}\log(ep)}{\sigma_{\min}}\left\{\frac{\nu_{1,n}\sqrt{\log(ep)}}{\sqrt{n}} + \frac{\nu_{3,n}(\log(ep))^{3/2}}{\underline{\sigma}^2\sqrt{n}}\right\}
    \end{split}
\end{equation}
We need to prove that the right hand side is smaller than ${\tau(\{X_i\}, \{Y_i\})}/{\sqrt{n}}.$
In order to show this, $C/\mathfrak{K}$ has to be less than $1$ and given the freedom to choose $\mathfrak{K}$ set $\mathfrak{K} = 2C$. Using this, the bound on $\mu(S_{1:n}(X), S_{1:n}(Y))$ simplifies as
\begin{equation}\label{eq:almost-final-induction-calculation}
    \begin{split}
        \mu(S_{1:n}(X), S_{1:n}(Y)) &\le 
        \frac{1}{2}\left[\frac{\tau(\{X_i\}, \{Y_i\})}{\sqrt{n}}\right] + \frac{C\nu_{1,n}\log(ep)\sqrt{\log(pn)}}{\sqrt{n}\sigma_{\min}} +
        \frac{1}{2pn\sqrt{\log(ep)}}\\
        &\quad\quad+ 
        \frac{C\nu_{3,n}\log^2(ep)\sqrt{\log(pn)}}{n^{1/2}\sigma_{\min}\underline{\sigma}^2}\log\left(1 + \frac{\underline{\sigma}^3\sqrt{n}}{2C\nu_{3,n}(\log(ep))^{3/2}}\right)\\ 
        &\quad\quad+ 
        \frac{1}{2pn(\log(ep))^{3/2}} 
        + \frac{C(\log(ep))^{3/2}}{n^{1/2}\underline{\sigma}^3}\nu_{3,n}
        + 
        \frac{2C^2(\log(ep))^{3/2}}{\sigma_{\min}}\left\{\frac{\nu_{1,n}}{\sqrt{n}} 
        + \frac{\nu_{3,n}\log(ep)}{\underline{\sigma}^2\sqrt{n}}\right\}
    \end{split}
\end{equation}
The quantity above  will be less than $\tau(\{X_i\}, \{Y_i\})/\sqrt{n}$ if and only if
\begin{align*}
\tau(\{X_i\}, \{Y_i\}) &\ge \frac{2C\nu_{1,n}\log(ep)\sqrt{\log(pn)}}{\sigma_{\min}} + \frac{2C\nu_{3,n}(\log(ep))^{3/2}}{\underline{\sigma}^3}\\
&\quad+ \frac{1}{p\sqrt{n\log(ep)}} + \frac{1}{p\sqrt{n}(\log(ep))^{3/2}}\\
&\quad+ \frac{4C\nu_{3,n}\log^2(ep)\sqrt{\log(pn)}}{\sigma_{\min}\underline{\sigma}^2}\log\left(1 + \frac{\underline{\sigma}^3\sqrt{n}}{2C\nu_{3,n}(\log(ep))^{3/2}}\right)\\
&\quad+ \frac{4C^2}{\sigma_{\min}}(\log(ep))^{3/2}\left\{\nu_{1,n} + \frac{\nu_{3,n}\log(ep)}{\underline{\sigma}^2}\right\}.
\end{align*}
Also, note that we require $\tau(\{X_i\}, \{Y_i\})$ to be larger than $1$ for the induction hypothesis to hold for $k = 1$. Hence adding $1$ to the right hand side of the above display (and further replacing  $1/(p\sqrt{n\log(ep)}) + 1/(p\sqrt{n}(\log(ep))^{3/2}) + 1$ on the right hand side by $3$) proves the result. \qed



\newpage
\appendix
\section[Repeat of Senatov's Result]{Theorem 5.1.1 of~\cite{senatov2011normal}}\label{appsec:repeat-Thm-Senatov}
In this section, we repeat the statement of Theorem 5.1.1 of~\cite{senatov2011normal}, a Berry--Esseen bound for multivariate random vectors. The notation is as follows: for any random vector $X$, $P_X$ denotes the probability measure of $X$. For any $\varepsilon > 0$, $\Phi_{\varepsilon}$ is the distribution of a Gaussian random variable with mean zero and variance $\varepsilon^2 I$ (scaled identity). The convolution of two probability measures $P, Q$ is denoted by $P*Q$. The equation numbers below are same those from~\cite{senatov2011normal}.
\begin{thm}[Theorem 5.1.1 of~\cite{senatov2011normal}]
Let a metric $\mu$ on the set of distributions in $\mathbb{R}^k$ possess the following properties:
\begin{enumerate}
    \item $\mu$ is regular and homogeneous of order $t \ge 0$, i.e., 
    \[
    \mu(P*R, Q*R) \le \mu(P, R)\quad\mbox{and}\quad \mu(P_{c*X}, P_{c*Y}) = c^t\mu(P_X, P_Y)\quad\mbox{for $c > 0$}.
    \]
    \item For any distribution $Q$, any normal law $\Phi$ with non-degenerate covariance operator, and any $\varepsilon > 0$,
    \begin{equation}\label{eq:smoothing-inequality-senatov}
    \mu(Q, \Phi) \le c\mu(Q*\Phi_{\varepsilon}, \Phi*\Phi_{\varepsilon}) + c_0(\Phi)\varepsilon,\tag{5.1.1}
    \end{equation}
    where $c$ is an absolute constant, $c_0(\Phi)$ depends only on $\Phi$.
    \item For any distributions $P, Q, U, V$,
    \begin{equation}\label{eq:variation-senatov}
    \mu(P * U, Q*U) \le \mu(U, V)\mathrm{Var}(P, Q) + \mu(P*V, Q*V).\tag{5.1.2}
    \end{equation}
    Here $\mathrm{Var}(P, Q) = \int_{\mathbb{R}^p} |P - Q|(dx)$ is (twice) the total variation distance. 
    \item For any distributions $P, Q$, 
    \begin{equation}\label{eq:zeta-mu-combination}
    \mu(P*\Phi_{\varepsilon}, Q*\Phi_{\varepsilon}) \le c\frac{\zeta_3(P, Q)}{\varepsilon^{3-t}}.\tag{5.1.3}
    \end{equation}
\end{enumerate}
Then for any $n\ge1$ and independent identically distributed random vectors $X_1, X_2, \ldots, X_n$ from $P$, the inequality
\[
\mu(P_{(X_1 + \ldots + X_n)/\sqrt{n}}), \Phi) ~\le~ c\frac{\mu(P, Q)}{n^{\gamma}} + c\frac{\zeta_3(P, Q)}{\sigma^{3-t}\sqrt{n}} + cc_0(\Phi)\left\{\frac{\zeta_1(P , Q)}{\sqrt{n}} + \frac{\zeta_3(P, Q)}{\sigma^2\sqrt{n}}\right\},
\]
is true, where $\Phi$ is a normal law with mean zero and covariance same as that of $(X_1 + \ldots + X_n)/\sqrt{n}$, $\gamma$ is an arbitrary fixed number, and the constants in the bound depend on $t, \gamma,$ and the constants in conditions 1 and 3 above.
\end{thm}
The four conditions here can be compared with Lemmas~\ref{lem:smoothing-inequality}--~\ref{lem:mu-zeta-combination} used in the proof of Theorem~\ref{thm:main-theorem-CLT}.

\section{Proof of Lemma~\ref{lem:smoothing-inequality}}\label{appsec:proof-of-lemma-smoothing-inequality}
Let $K(A) = \mathbb{P}(\varepsilon Z \in A)$. Then it is clear that there exists a constant $c > 0$ such that
\[
K(\{x:\|x\|_{\infty} \le c\varepsilon\sqrt{\log(ep)}\}) \ge 1 - 1/e,
\]
because the Gaussian concentration inequality (Lemma 3.1 of~\cite{MR2814399})
\[
\mathbb{P}(\|Z\|_{\infty} \ge \mbox{med}(\|Z\|_{\infty}) + t) \le e^{-t^2/2},
\]
and median of $\|Z\|_{\infty}$ is upper bounded by $c\sqrt{\log(ep)}$ yields,
\[
\mathbb{P}(\|Z\|_{\infty} \ge c\sqrt{\log(ep)} + \sqrt{2\log(e)}) \le \frac{1}{e}\quad\Rightarrow\quad K(\{x:\|x\|_{\infty} \le c\varepsilon\sqrt{\log(ep)}\}) \ge 1 - 1/e > 1/2.
\]
Now define probability measures $P_U(A) = \mathbb{P}(U\in A), P_V(A) = \mathbb{P}(V\in A)$. We now apply Lemma 2.4 of~\cite{fang2020high} with $\epsilon = c\varepsilon\sqrt{\log(ep)}$ to prove the result. Firstly, note that with $\epsilon = c\varepsilon\sqrt{\log(ep)}$, 
\[
K(\{x:\|x\|_{\infty} \le \epsilon\}) \ge 1 - 1/e > 1/2.
\]
Define the function
\[
h(x) := \mathbbm{1}\{x \preceq r\}.
\]
We want to upper bound $|\mathbb{P}(U \preceq r) - \mathbb{P}(V \preceq r)|$. This is exactly the same as $|\int hd(P_U - P_V)|$. We now compute $\gamma^*(h; \epsilon)$ and $\tau^*(h; 2\epsilon)$, in the notation of Lemma 2.4 of~\cite{fang2020high}. In their notation,
\[
M_{h_y}(x;\epsilon) = \sup_{z:\|z - x\|_{\infty} \le \epsilon} h_y(z) = \sup_{z:\|z - x\|_{\infty} \le \epsilon}\mathbbm{1}\{y + z \preceq r\} = \mathbbm{1}\{y + x \preceq r + \epsilon\mathbf{1}\}.
\]
Similarly,
\[
m_{h_y}(x; \epsilon) = \inf_{z:\|z - x\|_{\infty} \le \epsilon} h_y(z) = \mathbbm{1}\{y + x \preceq r - \epsilon\mathbf{1}\}.
\]
This implies
\begin{align*}
\tau(h_y; 2\epsilon) &= \int [\mathbbm{1}\{y + x \preceq r + \epsilon\mathbf{1}\} - \mathbbm{1}\{y + x\preceq r - \epsilon\mathbf{1}\}]dP_V(x)\\ 
&= \mathbb{P}(V \preceq r - y + \epsilon\mathbf{1}) - \mathbb{P}(V \preceq r - y + \epsilon\mathbf{1})\\
&\le \frac{C}{\sigma_{\min}}\epsilon\sqrt{\log(ep)},
\end{align*}
for any $y\in\mathbb{R}^p$. This follows from Nazarov's anti-concentration inequality~\citep{Naz03,chernozhukov2017detailed}. Therefore,
\begin{equation}\label{eq:tau-star-bound}
\tau^*(h; 2\epsilon) = \sup_{y\in\mathbb{R}^p}\tau(h_y; 2\epsilon) \le \frac{C}{\sigma_{\min}}\epsilon\sqrt{\log(ep)}.
\end{equation}
From the expression for $M_{h_y}(x;\epsilon)$, we write
\begin{align*}
\int M_{h_y}(x; \epsilon)d[(P_U - P_V)*K](x) &= \mathbb{E}[M_{h_y}(U + \varepsilon Z) - M_{h_y}(V + \varepsilon Z)]\\
&= \mathbb{P}(U + \varepsilon Z \preceq r + \epsilon\mathbf{1}) - \mathbb{P}(U + \varepsilon Z \preceq r + \epsilon\mathbf{1})\\
&\le \sup_{r\in\mathbb{R}^p}|\mathbb{P}(U + \varepsilon Z \preceq r) - \mathbb{P}(V + \varepsilon Z \preceq r)|.
\end{align*}
Similarly, 
\[
-\int m_{h_y}(x; \epsilon)d[(P_U - P_V)*K](x) \le \sup_{r\in\mathbb{R}^p}|\mathbb{P}(U + \varepsilon Z \preceq r) - \mathbb{P}(V + \varepsilon Z \preceq r)|.
\]
Therefore,
\begin{equation}\label{eq:gamma-star-boubd}
\gamma^*(h; \epsilon) = \sup_{y\in\mathbb{R}^d}\,\gamma(h_y; \epsilon) \le \sup_{r\in\mathbb{R}^d}|\mathbb{P}(U + \varepsilon Z \preceq r) - \mathbb{P}(V + \varepsilon Z \preceq r)|.
\end{equation}
Combining~\eqref{eq:tau-star-bound} and~\eqref{eq:gamma-star-boubd} yields the result.\qed


\section{Proof of Lemma~\ref{lem:TV-alternative-lemma}}\label{appsec:proof-of-lemma-TV-alternative-lemma}
Because $\varphi_{\varepsilon}(s, r) = \mathbb{P}(s + \varepsilon Z \preceq r)$, we have
\begin{align*}
\left|\mathbb{P}(U + \varepsilon Z \preceq r) - \mathbb{P}(V + \varepsilon Z \preceq r)\right| &= |\mathbb{E}[\varphi_{\varepsilon}(W + U, r) - \varphi_{\varepsilon}(W + V, r)]|\\
&= \left|\int \mathbb{E}\left[\left\langle\nabla \varphi_{\varepsilon}(W + \tau x, r), x\right\rangle(\mu_U - \mu_V)(dx)\right]\right|\\
&\le \int \mathbb{E}\left[\left\|\nabla \varphi_{\varepsilon}(W + \tau x, r)\right\|_1\right]\|x\|_{\infty}|\mu_U - \mu_V|(dx)
\end{align*}
Consider the event
\[
A(t, x) = \left\{r - t\mathbf{1} \preceq W + \tau x\preceq r + t\mathbf{1}\right\}.
\]
Note that
\[
\left\|\nabla \varphi_{\varepsilon}(W + \tau x, r)\right\|_1\mathbbm{1}\{A(t, x)\} \le \sup_{s,r}\|\nabla\varphi_{\varepsilon}(s, r)\|_1\mathbbm{1}\{A(t, x)\} \le \frac{c\sqrt{\log(ep)}}{\varepsilon}\mathbbm{1}\{A(t, x)\}.
\]
Furthermore, for $t = c\varepsilon\sqrt{\log(pn)}$ (for some constant $c$),
\[
\left\|\nabla \varphi_{\varepsilon}(W + \tau x, r)\right\|_1\mathbbm{1}\{A^c(t, x)\} \le \frac{c}{\varepsilon pn}.
\]
Therefore,
\begin{equation}\label{eq:penultimate-lemma-inequality}
    \left|\mathbb{P}(U + \varepsilon Z \preceq r) - \mathbb{P}(V + \varepsilon Z \preceq r)\right| \le \int \left[\frac{c\sqrt{\log(ep)}}{\varepsilon}\mathbb{P}(A(t, x)) + \frac{c}{\varepsilon pn}\right]\|x\|_{\infty}|\mu_U - \mu_V|(dx).
\end{equation}
Now, note that with $t = c\varepsilon\sqrt{\log(pn)}$,
\begin{align*}
\mathbb{P}(A(t, x)) &= \mathbb{P}(r - t\mathbf{1} \preceq W + \tau x \preceq r + t\mathbf{1})\\ 
&= \mathbb{P}(r - t\mathbf{1} \preceq W' + \tau x \preceq r + t\mathbf{1}) + \sup_{r\in\mathbb{R}^p}|\mathbb{P}(W \preceq r) - \mathbb{P}(W' \preceq r)|\\
&\le Ct\frac{\sqrt{\log(ep)}}{\sigma_{\min}} + \sup_{r\in\mathbb{R}^p}|\mathbb{P}(W \preceq r) - \mathbb{P}(W' \preceq r)|\\
&= C\varepsilon\frac{\sqrt{\log(pn)\log(ep)}}{\sigma_{\min}} + \sup_{r\in\mathbb{R}^p}|\mathbb{P}(W \preceq r) - \mathbb{P}(W' \preceq r)|.
\end{align*}
Combining the above two displays proves~\eqref{eq:zeta-1-bound}.

The proof for~\eqref{eq:zeta-3-bound} is almost the same, except we start with
\begin{align*}
\left|\mathbb{P}(W + U + \varepsilon Z \preceq r) - \mathbb{P}(W + V + \varepsilon Z \preceq r)\right| &= \left|\mathbb{E}[\varphi_{\varepsilon}(W + U, r) - \varphi_{\varepsilon}(W + V, r)]\right|\\
&= \left|\int \mathbb{E}\left[\left\langle \nabla^3(W + \tau x, r), x^{\otimes 3}\right\rangle(\mu_U - \mu_V)(dx)\right]\right|\\
&\le \int \mathbb{E}\left[\|\nabla^3\varphi_{\varepsilon}(W + \tau x, r)\|_1\right]\|x\|_{\infty}^3|\mu_U - \mu_V|(dx).
\end{align*}
Now again we have
\begin{align*}
    \|\nabla^3\varphi_{\varepsilon}(W + \tau x, r)\|_1\mathbbm{1}\{A(t, x)\} &\le \sup_{s, r}\|\nabla^3\varphi_{\varepsilon}(s, r)\|_1\mathbbm{1}\{A(t, x)\} \le \frac{c(\log(ep))^{3/2}}{\varepsilon^3}\mathbbm{1}\{A(t, x)\}\\
    \|\nabla^3\varphi_{\varepsilon}(W + \tau x, r)\|_1\mathbbm{1}\{A^c(t, x)\} &\le \frac{c}{\varepsilon^3pn}\quad\mbox{(for $t = c\varepsilon\sqrt{\log(pn)}$).}
\end{align*}
Therefore,
\begin{align*}
    |\mathbb{P}(W + U + \varepsilon Z \preceq r) - \mathbb{P}(W + V + \varepsilon Z \preceq r)| &\le \int \left[\frac{c(\log(ep))^{3/2}}{\varepsilon^3}\mathbb{P}(A(t, x)) + \frac{c}{\varepsilon^3pn}\right]\|x\|_{\infty}^3|\mu_U - \mu_V|(dx)\\
    &\le C\nu_3(U, V)\left[\frac{\log^2(ep)\sqrt{\log(pn)}}{\varepsilon^2\sigma_{\min}} + \frac{1}{\varepsilon^3pn}\right]\\ 
    &\qquad+ \frac{C\nu_3(U, V)(\log(ep))^{3/2}}{\varepsilon^3}\sup_{r\in\mathbb{R}^p}|\mathbb{P}(W \preceq r) - \mathbb{P}(W' \preceq r)|.
\end{align*}
The control of $\mathbb{P}(A(t, x))$ here is exactly the same as in the previous case. This completes the proof of~\eqref{eq:zeta-3-bound}.\qed


\section{Proof of Lemma~\ref{lem:regular-homogenity}}\label{appsec:proof-of-lemma-regular-homogenity}
The first one follows by writing
\[
\mathbb{P}(U + W \preceq r) - \mathbb{P}(V + W \preceq r) = \mathbb{E}[\mathbb{P}(U \preceq r - W|W) - \mathbb{P}(V \preceq r - W|W)] \le \sup_{r\in\mathbb{R}^p}|\mathbb{P}(U \preceq r) - \mathbb{P}(V \preceq r)|.
\]
The second one is obvious.\qed


\section{Proof of Lemma~\ref{lem:mu-zeta-combination}}\label{appsec:proof-of-lemma-mu-zeta-combination}
Note that 
\[
\mathbb{P}(U + \varepsilon Z \preceq r) - \mathbb{P}(V + \varepsilon Z \preceq r) = \mathbb{E}[\varphi_{\varepsilon}(U, r) - \varphi_{\varepsilon}(V, r)].
\]
Because $\varphi_{\varepsilon}(\cdot, r)$ is a thrice differentiable function and satisfies $\|\nabla^3\varphi_{\varepsilon}(x, r)\|_1 \le c\log^{3/2}(ep)/\varepsilon^3$ (by Theorem 3 of~\cite{bentkus1990smooth}), the result follows. Also, see Lemma 2.3 of~\cite{fang2020high} for a clear formulation of Theorem 3 of~\cite{bentkus1990smooth}. 
\section{Proof of Lemma~\ref{lem:ideal-to-pseudo}}\label{appsec:proof-of-lemma-ideal-to-pseudo}
From the definition~\eqref{eq:ideal-metric-3} of $\zeta_3(\cdot, \cdot)$, we have
\[
\zeta_3(c U, c V) = \sup_{f:\|\nabla^3f(x)\|_1 \le 1}\left|\mathbb{E}[f(c U)] - \mathbb{E}[f(c V)]\right|.
\]
Note that the function $g(x) = f(c x)$ satisfies $\|\nabla^3g(x)\|_1 = |c|^3\|\nabla^3f(x)\|_1$. This implies that $\widebar{g}(x) = f(c x)/|c|^3$ satisfies $\|\nabla^3\widebar{g}(x)\|_1 = \|\nabla^3f(x)\|_1$. Hence $\zeta_3(cU, cV) = |c|^3\zeta_3(U, V)$.

To prove the second inequality, it suffices to prove that $\zeta_3(U, V) \le \nu_3(U, V)/6$. For any function $f$ satisfying $\|\nabla^3f(x)\|_1 \le 1$ for all $x\in\mathbb{R}^p$, we have for a $\tau\sim U(0, 1)$, 
\begin{align*}
\mathbb{E}[f(U) - f(V)] &= \int f(x)(P_U - P_V)(dx) \\
&= \int \left\{f(0) + \langle\nabla f(0), x\rangle + \frac{1}{2}\langle\nabla^{2}f(0), x^{\otimes 2}\rangle + \frac{1}{2}\mathbb{E}[(1 - \tau)^2\langle \nabla^3f(\tau x), x^{\otimes 3}\rangle]\right\}(P_U - P_V)(dx)\\
&= \int \frac{1}{2}\mathbb{E}[(1 - \tau)^2\|\nabla^3f(\tau x)\|_1]\|x\|_{\infty}^3|P_U - P_V|(dx)\\
&\le \frac{1}{6}\int \|x\|_{\infty}^3|P_U - P_V|(dx).
\end{align*}
This completes the proof.
\section{Proof of Theorem~\ref{thm:optimal-dependence-sigma-min}}\label{appsec:proof-of-theorem-optimal-dependence-sigma-min}
Note that $\sigma_{\min} = \underline{\sigma}$ for a diagonal covariance matrix and in this case $\sigma_{\min}^{1/3}\underline{\sigma}^{2/3} = \sigma_{\min} = \underline{\sigma}$. For this reason, it suffices to exhibit a distribution with independent coordinates to prove the result. Consider the random vector $X$ with distribution given by
\[
X(j)\overset{iid}{\sim} N(0, 1)\quad\mbox{for}\quad 1\le j\le p-1,
\]
and $X(p)$ is independent of $X(1), \ldots, X(p-1)$ with distribution given by
\[
\mathbb{P}(X(p) = 0) = 1 - \frac{1}{\gamma},\quad\mbox{and}\quad \mathbb{P}(X(p) = -\gamma^{1/3}) = \mathbb{P}(X(p) = \gamma^{1/3}) = \frac{1}{2\gamma}. 
\]
Note that
\[
\sup_{r\in\mathbb{R}^p}\left|\mathbb{P}\left(\frac{1}{\sqrt{n}}\sum_{i=1}^n X_i \preceq r\right) - \mathbb{P}\left(\frac{1}{\sqrt{n}}\sum_{i=1}^n Y_i \preceq r\right)\right| \ge \frac{1}{2}\left|\mathbb{P}\left(\frac{1}{\sqrt{n}}\sum_{i=1}^n X_i(p) = 0\right) - 0\right| \ge \frac{1}{2}(1 - 1/\gamma)^n.
\]
In this case, we have 
\[
\sigma_{\min} = \underline{\sigma} = \gamma^{-1/6}.
\]
Furthermore, 
\begin{align*}
\nu_1(X, Y) &\le \mathbb{E}[\|X\|_{\infty} + \|Y\|_{\infty}] \lesssim \sqrt{\log(ep)}\\
\nu_3(X, Y) &\le \mathbb{E}[\|X\|_{\infty}^3 + \|Y\|_{\infty}^3] \le \mathbb{E}\left[\max_{1\le j \le p-1}|X(j)|^3\right] + \mathbb{E}\left[|X(p)|^3 + \|Y\|_{\infty}^3\right] \lesssim (\log(ep))^{3/2}.
\end{align*}
For the distribution specified above, with $\gamma = n$, the distribution of the average of $X_i$'s is bounded away from the Gaussian and hence we need 
\begin{equation}\label{eq:optimality-bound}
\frac{\nu_3(X, Y)}{\sqrt{n}\sigma^3(\mathbb{E}[XX^{\top}])}\log^{\alpha_2}(ep) + \frac{\nu_1(X, Y)}{\sqrt{n}\sigma(\mathbb{E}[XX^{\top}])}\log^{\alpha_1}(ep) \gtrsim 1. 
\end{equation}
The left hand side is upper bounded by
\[
\frac{(\log(ep))^{3/2 + \alpha_2}}{\sqrt{n}\sigma^3(\mathbb{E}[XX^{\top}])} + \frac{(\log(ep))^{1/2 + \alpha_1}}{\sqrt{n}\sigma(\mathbb{E}[XX^{\top}])}.
\]
Because the upper bound holds for any $p \ge 1$, it also holds for $p = 1$ and in this case~\eqref{eq:optimality-bound} is equivalent to
\[
\frac{1}{\sigma^3(\mathbb{E}[XX^{\top}])} + \frac{1}{\sigma(\mathbb{E}[XX^{\top}])} \gtrsim \sqrt{n}\quad\Leftrightarrow\quad \sigma(\mathbb{E}[XX^{\top}]) \lesssim n^{-1/6} = \sigma_{\min}^{1/3}\underline{\sigma}^{2/3} = \sigma_{\min} = \underline{\sigma}.
\]
This completes the proof.
\section{Proof of Theorem~\ref{thm:optimal-dependence-on-third-moment}}\label{appsec:proof-of-optimality-third-moment}
Consider random vectors $X_1, X_2, \ldots, X_n$ that are independent and identically distributed from the following distribution: $X_1(1), \ldots, X_1(p-1)\overset{iid}{\sim} N(0, 1)$ and $X_1(p)$ is independent of $X_1(j), 1\le j\le p-1$ and is distributed as 
\[
\mathbb{P}(X_1(p) = 0) = 1 - \frac{1}{\gamma},\,\mathbb{P}(X_1(p) = \gamma^{1/2}) = \mathbb{P}(X_1(p) = -\gamma^{1/2}) = \frac{1}{2\gamma}.
\]
It is clear that $\mathbb{E}[X_i(j)] = 0$ and $\mathbb{E}[X_i^2(j)] = 1$ for all $1\le j\le p$. Because of the independence, it follows that $\sigma_{\min} = \underline{\sigma} = 1$. 
As shown in the Appendix~\ref{appsec:proof-of-theorem-optimal-dependence-sigma-min} (the proof of Theorem~\ref{thm:optimal-dependence-sigma-min}), we have
\[
\mu\left(\frac{1}{\sqrt{n}}\sum_{i=1}^n X_i,\,\frac{1}{\sqrt{n}}\sum_{i=1}^n Y_i\right) \ge \frac{1}{2}(1 - 1/\gamma)^n,
\]
which can be lower bounded by a constant independent of $p, n$ if $\gamma = n$ for $n\ge2$. Observe now that
\[
\mathbb{E}[\|X_i\|_{\infty}^3] \le \mathbb{E}\left[\max_{1\le j\le p-1}|X_i(j)|^3\right] + \mathbb{E}[|X_i(p)|^3] \lesssim (\log(ep))^{3/2} + \gamma^{1/2},
\]
which is of order $n^{1/2}$ if $(\log(ep))^{3} \le n\sigma_{\max}^2\underline{\sigma}^4$. Therefore,
\[
C\frac{n^{-1}\sum_{i=1}^n \mathbb{E}[\|X_i\|_{\infty}^3]}{\sqrt{n}\sigma_{\min}\underline{\sigma}^2} \le (1 - 1/n)^n \le \mu\left(\frac{1}{\sqrt{n}}\sum_{i=1}^n X_i, \frac{1}{\sqrt{n}}\sum_{i=1}^n Y_i\right).
\]

\bibliography{HDCLT}
\bibliographystyle{apalike}
\end{document}